\documentclass[12pt,a4paper]{article}
\RequirePackage{amsthm,amsmath,amssymb,amsfonts}

\theoremstyle{definition}

\newtheorem{remark}{Remark}
\theoremstyle{plain}
\newtheorem{theorem}{Theorem}
\newtheorem{lemma}{Lemma}

\author{Shoou-Ren Hsiau, Ting-Yi Tsai, and Yi-Ching Yao\\
National Changhua  University of Education and Academia Sinica}

\title{Stochastic dominance of first return times for nearest-neighbor random 
walks on $\mathbb{Z}^d$}

\begin{document}

\newcommand{\ML}{\mathcal{L}}
\newcommand{\MP}{\mathbb{P}}
\newcommand{\MS}{\mathbb{S}}
\newcommand{\MST}{\widetilde{\mathbb{S}}}
\newcommand{\MSTU}{\widetilde{\mathbb{S}^{+}}}
\newcommand{\MSTD}{\widetilde{\mathbb{S}^{-}}}
\newcommand{\MSO}{\widehat{\mathbb{S}}}
\newcommand{\MSOU}{\widehat{\mathbb{S}^+}}
\newcommand{\MSOD}{\widehat{\mathbb{S}^-}}
\newcommand{\HT}{\widehat{T}}
\newcommand{\CS}{\mathcal{S}}
\newcommand{\MM}{\mathcal{M}}
\newcommand{\MG}{\mathcal{G}}
\newcommand{\half}{\frac{1}{2}}
\newcommand{\Sh}{S^{\mathbf{h}}}
\newcommand{\Xh}{X^{\mathbf{h}}}
\newcommand{\Sha}{S^{\mathbf{h}'}}
\newcommand{\Shb}{S^{\mathbf{h}''}}
\newcommand{\Xha}{X^{\mathbf{h}'}}
\newcommand{\Xhb}{X^{\mathbf{h}''}}
\newcommand{\Sr}{S^{\mathbf{r}}}
\newcommand{\Sra}{S^{\mathbf{r}'}}
\newcommand{\Srb}{S^{\mathbf{r}''}}
\newcommand{\zero}{\mathbf{0}}
\newcommand{\bfe}{\mathbf{e}}
\newcommand{\Th}{T^{\mathbf{h}}}
\newcommand{\Tha}{T^{\mathbf{h}'}}
\newcommand{\Thb}{T^{\mathbf{h}''}}
\newcommand{\ph}{p^{\mathbf{h}}}
\newcommand{\qh}{q^{\mathbf{h}}}
\newcommand{\qha}{q^{\mathbf{h}'}}
\newcommand{\gh}{\theta^{\mathbf{h}}}
\newcommand{\Ph}{P^{\mathbf{h}}}
\newcommand{\Qh}{Q^{\mathbf{h}}}
\newcommand{\fh}{f^{\mathbf{h}}}
\newcommand{\Sa}{S^{\alpha}}
\newcommand{\Xa}{X^{\alpha}}
\newcommand{\Ta}{T^{\alpha}}
\newcommand{\pa}{p^{\alpha}}
\newcommand{\qa}{q^{\alpha}}
\newcommand{\Pa}{P^{\alpha}}
\newcommand{\Qa}{Q^{\alpha}}
\newcommand{\fa}{f^{\alpha}}
\newcommand{\Da}{\text{D}_{\alpha}}
\newcommand{\ga}{\theta^{\alpha}}
\newcommand{\Sax}{S^{\alpha, \xi}}
\newcommand{\Xax}{X^{\alpha, \xi}}
\newcommand{\Tax}{T^{\alpha, \xi}}
\newcommand{\pax}{p^{\alpha, \xi}}
\newcommand{\qax}{q^{\alpha, \xi}}
\newcommand{\Pax}{P^{\alpha, \xi}}
\newcommand{\Qax}{Q^{\alpha, \xi}}
\newcommand{\fax}{f^{\alpha, \xi}}
\newcommand{\gax}{\theta^{\alpha,\xi}}

\maketitle
\begin{abstract}
For a $d$-dimensional probability vector $\mathbf{h}=(h_1,\dots, h_d)$, let 
$(\Sh_n)_{n\geq 0}$ be a  nearest-neighbor random walk on $\mathbb{Z}^d$ such that 
at each step, it moves to one of the two nearest neighbors in the $i$-th dimension
with probability $\half h_i$ ($i=1,\dots, d$). Let $\Th=\inf\{n\geq 1: 
\Sh_n=(0,\dots,0)\}$, the first return time to the origin. For two 
$d$-dimensional probability
vectors $\mathbf{h}'$ and $\mathbf{h}''$ with the former majorizing the latter,
we show that $\Tha$ is stochastically smaller than $\Thb$. In particular,
the first return time for the $d$-dimensional simple random walk is stochastically
larger than $\Th$ for all $d$-dimensional probability vectors $\mathbf{h}$.


\end{abstract}

\section{Introduction and the main result}

For the $d$-dimensional simple random walk, let $\tau_d$ denote the first return time to 
its starting point.  
The classic P\'olya's random walk theorem (\cite{Polya}) states that the $d$-dimensional simple random walk is
recurrent for $d\leq 2$ and transient for $d\geq 3$, i.e. the return probability 
$\MP(\tau_d<\infty)=1$  for $d\leq 2$ and $\MP(\tau_d<\infty)<1$ for $d\geq 3$.  We will show that 
$\tau_d$ is stochastically increasing in $d$, i.e. $\MP(\tau_d\leq n)$ is decreasing 
in $d \geq 1$ for all 
$n$.  In particular, the escape probability $\MP(\tau_d=\infty)=1-\MP(\tau_d<\infty)$ is increasing in $d$. In the literature, some early works 
\cite{Daley, Lehman, Montroll1956, Montroll1964, Montroll1965} considered more general random walks and computed
the return probability for the 3-dimensional case.

For two $d$-dimensional vectors $\mathbf{h}'=(h_1',\dots, h_d')$ and
$\mathbf{h}''=(h_1'',\dots, h_d'')$,  $\mathbf{h}'$ is said (\emph{cf.} \cite{MOA}) to majorize $\mathbf{h}''$
(written  $\mathbf{h}' \succeq \mathbf{h}''$) if 
\begin{align*}
\sum_{i=1}^j h_{(i)}' \geq \sum_{i=1}^j h_{(i)}''\;\;(j=1,\dots, d-1)\;\;\text{and}
\;\;\sum_{i=1}^d h_{(i)}' = \sum_{i=1}^d h_{(i)}'',
\end{align*}
where $(h_{(1)}',\dots, h_{(d)}')$ is a permutation of $(h_1',\dots, h_d')$ in decreasing order
and similarly for $(h_{(1)}'',\dots, h_{(d)}'')$. We say  $\mathbf{h}'$ strictly
majorizes $\mathbf{h}''$ and 
  write $\mathbf{h}' \succ \mathbf{h}''$
if $\mathbf{h}' \succeq \mathbf{h}''$ and $\mathbf{h}'$ is not a permutation of $\mathbf{h}''$.
 A vector 
$\mathbf{h}=(h_1,\dots, h_d)$ is called a probability vector if $h_i\geq 0$ for all $i$
and $h_1+\cdots+h_d=1$.
Let $\zero_d=(0,\dots,0)$ denote the $d$-dimensional vector of zeros, and $\bfe_d^{i}$  
 the standard $i$-th unit vector in $\mathbb{Z}^d$, $i=1,\dots, d$.
For a probability vector $\mathbf{h}=(h_1,\dots,h_d)$, let
$(\Sh_n)_{n\geq 0}$ be a random walk on the $d$-dimensional integer lattice $\mathbb{Z}^d$ such that
$\Sh_0=\zero_d$ and $\Sh_n=\Xh_1+\cdots+\Xh_n \;(n\geq 1)$, where the 
independent and identically distributed (i.i.d.) increments $\Xh_n$ satisfy
$\MP(\Xh_n=\bfe_d^{i})=\MP(\Xh_n=-\bfe_d^{i})=\half h_i$, $i=1,\dots, d$. Let
$\Th=\inf\{n>0: \Sh_n=\zero_d\}$, the first return time to $\zero_d$. Theorem \ref{Thm1} below is
the main result of this paper.

\begin{theorem}\label{Thm1}
For two probability vectors $\mathbf{h}'$ and $\mathbf{h}''$ with $\mathbf{h}' \succ \mathbf{h}''$, 
we have that $\Tha$ is stochastically smaller than $\Thb$, i.e. 
$\MP(\Tha \leq n)\geq \MP(\Thb \leq n)$ for $n\geq 1$.
\end{theorem}

The next section  introduces further notation and presents preliminary results (including Lemma 1 which
plays a crucial role in the proof of Theorem \ref{Thm1}). Sections 3 and 4 prove Theorem \ref{Thm1}
for the cases $d=2$ and  $d\geq 3$, respectively.

\begin{remark}
The class of random walks $(\Sh_n)_{n\geq 0}$ (indexed by a probability vector
$\mathbf{h}$)  consists of all nearest-neighbor random walks with the
additional symmetry property that the increments $\Xh_n$ have the same distribution as $-\Xh_n$.
In \cite{Montroll1956}, this class of random walks was introduced and studied for some
special cases of $\mathbf{h}$ with an emphasis on the return probability.
Note that the distribution of $\Th$ is invariant with
respect to permutations of $\mathbf{h}=(h_1,\dots, h_d)$. To the best of our knowledge,
there has been little study of first return time. Neither of the classic books 
\cite{LL, Spitzer}  discusses first return time.
\end{remark}

\begin{remark}
The $d$-dimensional probability vectors $(1,0,\dots, 0), (\half, \half, 0,\dots, 0),$ $\dots$, and 
$(\frac{1}{d}, \frac{1}{d},\dots, \frac{1}{d})$ 
are decreasing in the sense of majorization.  
 Setting $\mathbf{h}$ equal to each of the above probability vectors,
the random walks 
$(\Sh_n)_{n\geq 0}$ 
correspond, respectively,  to
the $1$-, $2$-, $\dots$, $d$-dimensional simple random walks. Consequently, $\Th=\tau_d$  for
$\mathbf{h}= (\frac{1}{d}, \frac{1}{d},\dots, \frac{1}{d})$. It follows that
$\tau_d$ is stochastically increasing in $d$. Also $\tau_d$ is stochastically larger than
$\Th$ for all $d$-dimensional probability vectors $\mathbf{h} \neq
(\frac{1}{d}, \frac{1}{d},\dots, \frac{1}{d})$  since all such $\mathbf{h} \succ
(\frac{1}{d}, \frac{1}{d},\dots, \frac{1}{d})$. 
\end{remark}

\begin{remark}
By \cite[Theorem 2, Section 3, Chapter III]{Feller}, we have
$\MP(\Th<\infty)=1-1/(1+R^{\mathbf{h}})$, where 
$R^{\mathbf{h}}=\sum_{n=1}^\infty \MP(\Sh_{2n}=\zero_d)$ 
(the expected total number of returns to $\zero_d$). So $\MP(\Th<\infty)=1$ if and only if
$R^{\mathbf{h}}=\infty$.
 By Theorem \ref{Thm1}, 
$\MP(\Tha<\infty)\geq \MP(\Thb<\infty)$ if $\mathbf{h}' \succ \mathbf{h}''$.
If $\mathbf{h}' \succ \mathbf{h}''$ and $\mathbf{h}''$
 has fewer than $3$ non-zero entries (i.e. $(\Shb_n)_{n\geq 0}$ is genuinely
a $1$- or $2$-dimensional random walk), then $\MP(\Tha<\infty)=\MP(\Thb<\infty)=1$.
 By Remark \ref{section4} in Section 4, if
$\mathbf{h}' \succ \mathbf{h}''$, 
$\MP(\Sha_{2n}=\zero_d)>\MP(\Shb_{2n}=\zero_d)$ for $n\geq 1$. Thus 
 if $\mathbf{h}' \succ \mathbf{h}''$
and  $\mathbf{h}''$ has three or  more non-zero entries, 
we have $R^{\mathbf{h}''}<\infty$ and  $R^{\mathbf{h}'}>R^{\mathbf{h}''}$,
implying that $\MP(\Tha<\infty)>\MP(\Thb<\infty)$. In particular,
$\MP(\tau_d <\infty)$ is strictly decreasing in $d \geq 3$.
Furthermore, since
$\MP(\Sh_2=\zero_d)=\MP(\Th\leq 2)$,
we have $\MP(\Tha\leq 2)>\MP(\Thb \leq 2)$ for $\mathbf{h}'\succ \mathbf{h}''$, 
implying that $\Tha$ is strictly  stochastically smaller
than $\Thb$. So $\tau_d$ is strictly stochastically increasing in $d \geq 1$.
\end{remark}

\begin{remark}
The return probability of the $d$-dimensional simple walk is given (\emph{cf.} \cite[(5.1), p. 246]{Montroll1956}) by 
\[
\MP(\tau_d<\infty)=1-\Big[\int_0^\infty e^{-x} \{I_0(x/d)\}^d \text{d}x\Big]^{-1},
\]
where $I_0(x)=\sum_{m=0}^\infty \frac{1}{(m!)^2} (\frac{x}{2})^{2m}$ is the modified Bessel function of order 0. For $d=3,4,5,6$, $\MP(\tau_d< \infty)\approx 0.34, 0.20, 0.13$ and $0.10$ 
(\emph{cf.} \cite[lines 5-6, p. 247]{Montroll1956}). Furthermore, by \cite[(5.4), p. 247]{Montroll1956},
$\MP(\tau_d<\infty)$ has the asymptotic expansion $\frac{1}{2d}\Big\{1+\frac{2}{2d}+\frac{7}{(2d)^2}+
\cdots\Big\}$ for large $d$. 
\end{remark}

\section{Preliminaries}

For a  random walk $(\Sh_n)_{n\geq 0}$, let $\ph_n=\MP(\Th=n)$, the probability that the random walk returns to $\zero_d$ for the first time after $n$ steps. Let 
$\qh_n=\MP(\Sh_n=\zero_d)$. Note that $\ph_0=0, \qh_0=1$, and
$\ph_n=\qh_n=0$ for odd $n$. Also $\ph_n>0$ for all even $n\geq 2$ and
$\qh_n>0$ for all even $n\geq 0$. It is readily seen (\emph{cf.} \cite[p. 712]{Novak}) that
\begin{align}\label{eq1}
\qh_{2n}=\sum_{k=1}^n \ph_{2k}\; \qh_{2n-2k}\;\; \text{for}\;\; n\geq 1, 
\end{align}
which is equivalent to
\begin{align}\label{eq2}
\Ph(z) \Qh(z)=\Qh(z)-1,
\end{align}
where 
\begin{align*}
\Ph(z)=\sum_{n=0}^\infty \ph_{2n}\; z^{n}\;\;\text{and}\;\; 
\Qh(z)=\sum_{n=0}^\infty \qh_{2n}\; z^{n},
\end{align*}
 the generating functions of the sequences
$(\ph_{2n}, n=0,1,\dots)$ and $(\qh_{2n}, n=0,1,\dots)$.
Furthermore, 
\begin{align}\label{eq2.99}
(\Qh(z))^2=\sum_{n=0}^\infty \gh_n\; z^n\;\;\text{with}\;\; \gh_n=\sum_{k=0}^n \qh_{2k}\; \qh_{2n-2k},
\;n\geq 0.
\end{align}
(Note that since $\ph_n=\qh_n=0$ for odd $n$,  the generating functions
of the sequences $(\ph_n, n=0,1,\dots)$ and $(\qh_n, n=0,1,\dots)$ are 
$\Ph(z^2)$ and $\Qh(z^2)$.)

For $n\geq 1$ and $1\leq j\leq d$, let
\begin{align}
\Gamma_n&=\{(\ell_1,\dots,\ell_d): \ell_i \geq 0 \;\text{for}\;i\geq 1 , \ell_1+\cdots+\ell_d=n\},
\label{eq2.02}\\
\Gamma_{n, j}'
 &=\{(\ell_1,\dots,\ell_d)\in \Gamma_n: \ell_j\geq 1\}.\label{eq2.03}
\end{align}
For either of the events $\{\Sh_{2n}=\zero_d\}$ and $\{\Sh_{2n}=2 \bfe_d^{j}\}$ to occur,  in the first $2n$ steps, 
 the random walk must  have moved in each dimension for an even number of times.  We have
\begin{align}
\qh_{2n}=\MP(\Sh_{2n}=\zero_d)
&=\sum_{(\ell_1,\dots,\ell_d) \in \Gamma_n} \frac{(2n)!}{\prod_{i=1}^d (2\ell_i)!} \prod_{i=1}^d
\binom{2\ell_i}{\ell_i} \prod_{i=1}^d (\half h_i)^{2\ell_i}\notag\\
&=\sum_{(\ell_1,\dots,\ell_d) \in \Gamma_n} \frac{(2n)!}{\prod_{i=1}^d (\ell_i!)^2} 
 \prod_{i=1}^d (\half h_i)^{2\ell_i}\;,\label{eq2.00}\\
\MP(\Sh_{2n}=2\bfe_d^{j})&=\sum_{(\ell_1,\dots,\ell_d) \in \Gamma_{n,j}'} 
\frac{(2n)!}{\prod_{i=1}^d (2\ell_i)!} \Big(\frac{(2\ell_j)!}{(\ell_j-1)!(\ell_j+1)!}\Big)
\prod_{1\leq i \leq d, i\neq j} 
\binom{2\ell_i}{\ell_i} \prod_{i=1}^d (\half h_i)^{2\ell_i}\notag\\
&=\sum_{(\ell_1,\dots,\ell_d) \in \Gamma_{n,j}'} 
\frac{(2n)!}{(\ell_j-1)!(\ell_j+1)!
\prod_{1\leq i \leq d, i\neq j} 
(\ell_i!)^2} \prod_{i=1}^d (\half h_i)^{2\ell_i}.\label{eq2.01}
\end{align}
Let 
\begin{align}\label{eq3}
\fh_{n}=2\sum_{1\leq i<j\leq d} h_i h_j \;\MP(\Sh_{n}=\bfe_d^{i}+\bfe_d^j)\;\;\text{for}\;\;n\geq 0.
\end{align}
Note that $\fh_0=0$ and $\fh_n=0$ for odd $n$. 
The next lemma is a key result in the proof of Theorem \ref{Thm1}.
\begin{lemma}\label{Lem1}
For $n\geq 0$, we have
\begin{align}
\frac{2n+2}{2n+1}\; \qh_{2n+2}=2\;\qh_2\;\qh_{2n}+ \fh_{2n}.\notag
\end{align}
\end{lemma}
\begin{proof}
The lemma holds trivially for $n=0$. 
For $n \geq 1$, we have
\begin{align}\label{eq4}
\qh_{2n+2}=\MP(\Sh_{2n+2}=\zero_d)=\MP(\Sh_2 \in \Lambda, \Sh_{2n+2}=\zero_d)
=A_n+B_n+\sum_{i=1}^d C_{n,i},
\end{align}
where $\Lambda=\{\zero_d\}\cup 
\{\pm \bfe_d^{i} \pm \bfe_d^j:  1\leq i<j\leq d\} \cup \{\pm 2 \bfe_d^{i}: i=1,\dots, d\}$,
and 
\begin{align*}
A_n&=\MP(\Sh_2=\zero_d, \Sh_{2n+2}=\zero_d),\\
B_n&= \sum_{1\leq i< j \leq d}\; \sum_{u, v \in \{0,1\}}
\MP(\Sh_2=(-1)^u\bfe_d^{i}+(-1)^v \bfe_d^j, \Sh_{2n+2}=\zero_d),\\
C_{n,i}&= \sum_{u \in \{0,1\}} \MP(\Sh_2= (-1)^u 2 \bfe_d^{i}, \Sh_{2n+2}=\zero_d),\; 
i=1,\dots, d.
\end{align*}
Then
\begin{align}\label{eq5}
A_n=\MP(\Sh_2=\zero_d)\; \MP(\Sh_{2n+2}=\zero_d \; \big| \; \Sh_2=\zero_d)=\qh_2\;\qh_{2n}.
\end{align}
By symmetry, for $1\leq i<j\leq d$ and $u, v \in \{0,1\}$,
\begin{align*}
\MP(\Sh_2=(-1)^u \bfe_d^{i}&+(-1)^v \bfe_d^j, \Sh_{2n+2}=\zero_d)\\
&=\MP(\Sh_2=-\bfe_d^{i}-\bfe_d^j, \Sh_{2n+2}=\zero_d)\\
&=\MP(\Sh_2=-\bfe_d^{i}-\bfe_d^j)\;\MP(\Sh_{2n+2}=\zero_d\;\big|\;\Sh_2=-\bfe_d^{i}-\bfe_d^j)\\
&=2 (\half h_i) (\half h_j)\;\MP(\Sh_{2n+2}=\bfe_d^{i}+\bfe_d^j\;\big|\;\Sh_2=\zero_d)\\
&=\half h_i h_j\;\MP(\Sh_{2n}=\bfe_d^{i}+\bfe_d^j),
\end{align*}
so that
\begin{align}
B_n&= \sum_{1\leq i< j \leq d}\; \sum_{u,v \in\{0,1\}}
\MP(\Sh_2=(-1)^u \bfe_d^{i}+(-1)^v \bfe_d^j, \Sh_{2n+2}=\zero_d)\notag\\
&=2\; \sum_{1\leq i<j\leq d} h_i h_j \MP(\Sh_{2n}=\bfe_d^{i}+\bfe_d^j)=\fh_{2n}.\label{eq6}
\end{align}
By symmetry again,
\begin{align*}
\MP(\Sh_2=2 \bfe_d^{i}, \Sh_{2n+2}=\zero_d)&=\MP(\Sh_2=-2 \bfe_d^{i}, \Sh_{2n+2}=\zero_d)\\
&=\MP(\Sh_2=-2 \bfe_d^{i}) \;\MP(\Sh_{2n+2}=\zero_d \;\big|\; \Sh_2=-2 \bfe_d^{i})\\
&=\MP(\Sh_2=-2 \bfe_d^{i}) \;\MP(\Sh_{2n+2}=2 \bfe_d^{i} \;\big|\; \Sh_2= \zero_d)\\
&=(\half h_i)^2 \;\MP(\Sh_{2n}=2 \bfe_d^{i}),
\end{align*}
so that
\begin{align}\label{eq7}
C_{n,i}= \sum_{u \in \{0,1\}} \MP(\Sh_2=(-1)^u 2 \bfe_d^{i}, \Sh_{2n+2}=\zero_d)=\half  h_i^2\; \MP(\Sh_{2n}=2 \bfe_d^{i}).
\end{align}
Recalling the definition of $\Gamma_n$ in (\ref{eq2.02}),
we claim that
\begin{align}\label{eq8}
C_{n,i}=D_{n,i}-E_{n,i},\; i=1,\dots, d,
\end{align}
where 
\begin{align}
D_{n,i}&=2 \sum_{(\ell_1,\ell_2,\dots,\ell_d) \in \Gamma_{n+1}} \frac{(2n)! \;(\ell_i)^2}
{\prod_{j=1}^d (\ell_j!)^2}  \prod_{j=1}^d (\half h_j)^{2 \ell_j}\label{eq8.1}\\
&=\half h_i^2\; \qh_{2n},\label{eq9}\\
E_{n,i}&=2 \sum_{(\ell_1,\ell_2,\dots,\ell_d) \in \Gamma_{n+1}} \frac{(2n)!\;\ell_i}
{\prod_{j=1}^d (\ell_j!)^2}  \prod_{j=1}^d (\half h_j)^{2 \ell_j}\;.\label{eq10}
\end{align}
To show (\ref{eq8}), for notational simplicity,  we treat only the case $i=1$.
Recall the definition of $\Gamma_{n,j}'$ in (\ref{eq2.03}) and  let  
$\Gamma_{n, 1}''=\{(\ell_1,\dots,\ell_d)\in \Gamma_n: \ell_1\geq 2\}$,
so that $\Gamma_{n,1}'' \subset \Gamma_{n,1}' \subset \Gamma_{n}$. We have by (\ref{eq2.01}) 
and (\ref{eq7})
\begin{align*}
C_{n,1}&=\half h_1^2 \;\MP(\Sh_{2n}=2 \bfe_d^{1})\\
&=\half h_1^2 \sum_{(\ell_1,\dots, \ell_d)\in \Gamma_{n, 1}'} \frac{(2n)!}{(\ell_1-1)! (\ell_1+1)!
\prod_{j=2}^d (\ell_j!)^2}
\prod_{j=1}^d (\half h_j)^{2\ell_j}\\
&=\half h_1^2 \sum_{(\ell_1,\dots, \ell_d)\in \Gamma_{n, 1}'} \frac{(2n)! \;\ell_1 (\ell_1+1)}{((\ell_1+1)!)^2
\prod_{j=2}^d (\ell_j!)^2}
\prod_{j=1}^d (\half h_j)^{2\ell_j}\\
&=2 \sum_{(\ell_1'-1,\ell_2,\dots,\ell_d) \in \Gamma_{n, 1}'} \frac{(2n)! (\ell_1'-1) \ell_1'}
{(\ell_1'!)^2 \prod_{j=2}^d (\ell_j!)^2} (\half h_1)^{2\ell_1'} \prod_{j=2}^d (\half h_j)^{2 \ell_j}
\;\;(\text{setting}\;\; \ell_1'=\ell_1+1)\\
&=2 \sum_{(\ell_1',\ell_2,\dots,\ell_d) \in \Gamma_{n+1, 1}''} \frac{(2n)! (\ell_1'-1) \ell_1'}
{(\ell_1'!)^2 \prod_{j=2}^d (\ell_j!)^2} (\half h_1)^{2 \ell_1'} \prod_{j=2}^d (\half h_j)^{2 \ell_j}\\
&=2 \sum_{(\ell_1,\ell_2,\dots,\ell_d) \in \Gamma_{n+1}} \frac{(2n)! (\ell_1-1) \ell_1}
{\prod_{j=1}^d (\ell_j!)^2}  \prod_{j=1}^d (\half h_j)^{2 \ell_j}\;\;
(\text{terms with}\;\; \ell_1=0,1\;\;\text{vanishing})\\
&=D_{n,1}-E_{n,1},
\end{align*}
establishing (\ref{eq8}) for $i=1$ (by (\ref{eq8.1}) and (\ref{eq10})). It remains to show
(\ref{eq9}).
Again for $i=1$, we have 
\begin{align*}
D_{n,1}&=2 \sum_{(\ell_1,\ell_2,\dots,\ell_d) \in \Gamma_{n+1}} \frac{(2n)! \;(\ell_1)^2}
{\prod_{j=1}^d (\ell_j!)^2}  \prod_{j=1}^d (\half h_j)^{2 \ell_j}\\
&=2 \sum_{(\ell_1,\ell_2,\dots,\ell_d) \in \Gamma_{n+1,1}'} \frac{(2n)! \;(\ell_1)^2}
{\prod_{j=1}^d (\ell_j!)^2}  \prod_{j=1}^d (\half h_j)^{2 \ell_j}\\
&=2 \sum_{(\ell_1-1,\ell_2,\dots,\ell_d) \in \Gamma_{n}} \frac{(2n)!}
{[(\ell_1-1)!]^2 \prod_{j=2}^d (\ell_j!)^2}  \prod_{j=1}^d (\half h_j)^{2 \ell_j}\\
&=2 (\half h_1)^2 \sum_{(\ell_1-1,\ell_2,\dots,\ell_d) \in \Gamma_{n}} \frac{(2n)!}
{[(\ell_1-1)!]^2 \prod_{j=2}^d (\ell_j!)^2} (\half h_1)^{2(\ell_1-1)}
 \prod_{j=2}^d (\half h_j)^{2 \ell_j}\\
&=2 (\half h_1)^2 \sum_{(\ell_1,\ell_2,\dots,\ell_d) \in \Gamma_{n}} \frac{(2n)!}
{\prod_{j=1}^d (\ell_j!)^2}  \prod_{j=1}^d (\half h_j)^{2 \ell_j}\;\;(\text{replacing}\;\;\ell_1-1\;\;
\text{by}\;\; \ell_1)\\
&=\half h_1^2 \;\MP(\Sh_{2n}=\zero_d)=\half h_1^2\;\qh_{2n}\;\;
(\text{by} \;\;(\ref{eq2.00})),
\end{align*}
proving (\ref{eq9}).
Now by (\ref{eq9}) and (\ref{eq10}),
\begin{align}
\sum_{i=1}^d D_{n,i}&=\sum_{i=1}^d \half h_i^2 \; \qh_{2n}
= \MP(\Sh_2=\zero_d)\; \qh_{2n}
=\qh_2\;\qh_{2n},\label{eq11}\\
\sum_{i=1}^d E_{n,i}&=\sum_{i=1}^d \;2
\sum_{(\ell_1,\ell_2,\dots,\ell_d) \in \Gamma_{n+1}}  \frac{(2n)!\;\ell_i}
{\prod_{j=1}^d (\ell_j)^2}  \prod_{j=1}^d (\half h_j)^{2 \ell_j}\notag\\
&=2(n+1) \sum_{(\ell_1,\ell_2,\dots,\ell_d) \in \Gamma_{n+1}}  \frac{(2n)!}
{\prod_{j=1}^d (\ell_j)^2}  \prod_{j=1}^d (\half h_j)^{2 \ell_j}\notag\\
&=\frac{2(n+1)}{(2n+1)(2n+2)} \sum_{(\ell_1,\ell_2,\dots,\ell_d) \in \Gamma_{n+1}}  \frac{(2n+2)!}
{\prod_{j=1}^d (\ell_j)^2}  \prod_{j=1}^d (\half h_j)^{2 \ell_j}\notag\\
&=\frac{1}{2n+1} \MP(\Sh_{2n+2}=\zero_d)=\frac{1}{2n+1}\; \qh_{2n+2},\label{eq12}
\end{align}
where the second-to-last equality is by (\ref{eq2.00}).
By (\ref{eq4})--(\ref{eq8}) and  (\ref{eq11})--(\ref{eq12}), we have
\begin{align*}
\qh_{2n+2}=2 \qh_2\;\qh_{2n}-\frac{1}{2n+1}\; \qh_{2n+2}+\fh_{2n},
\end{align*}
from which the lemma follows.
\end{proof}

\section{The case $d=2$}

In this section, we prove Theorem \ref{Thm1} for the case $d=2$. 
 Let $\mathbf{h}=(\alpha, 1-\alpha)$
where $\alpha \in [0,1]$. We write $\Sh_n=S_n^{\alpha}, \ph_{n}=\pa_{n}, \qh_n=\qa_n,
\gh_n=\ga_n,  \fh_n=\fa_n$, and
\begin{align*}
\Ph(z)&=\Pa(z)=\sum_{n=1}^\infty \pa_{2n}\; z^{n},\;\;\; \Qh(z)=\Qa(z)=\sum_{n=0}^\infty \qa_{2n}\; z^{n}\;.
\end{align*}
 We also write $\Da=\frac{\text{d}}{\text{d} \alpha}$. For 
$\alpha', \alpha'' \in [0,\half]$, $(\alpha', 1-\alpha') \succ (\alpha'',1-\alpha'')$
if and only if $\alpha' < \alpha''$. Thus for $d=2$, Theorem \ref{Thm1} is equivalent to
the statement that
\begin{align}\label{eqd=2}
 \sum_{k=1}^n \pa_{2k}\;\;\text{is decreasing in}\;\;\alpha
\in [0,\half]\;\;\text{for}\;\; n\geq 1.
\end{align} 
 Throughout, we take the convention that $\sum_{k=\ell}^m :=0$ for $m<\ell$ and $z$ is arbitrary
with $|z|<1$. Since $(1-\Pa(z)) \Qa(z)=1$  by (\ref{eq2}),
we have
\begin{align}\label{eq13.01}
\Da (1-\Pa(z))=\Da(1/\Qa(z))=-\Da \Qa(z)/(\Qa(z))^2,
\end{align}
where the operator $\Da$ treats $z$ as fixed.
By (\ref{eq2.99}) and (\ref{eq13.01}),
\begin{align}\label{eq14}
\sum_{n=1}^\infty \Da \pa_{2n}\; z^{n}&=\frac{\sum_{n=1}^\infty \Da \qa_{2n}\; z^{n}}{(\Qa(z))^2}
=\frac{\sum_{n=1}^\infty \Da \qa_{2n}\; z^{n}}{\sum_{n=0}^\infty \ga_{n}\;z^{n}}\;.
\end{align}

\begin{remark}
Note that $\pa_{2n}=K_n \;\alpha^n (1-\alpha)^n$ and $\qa_{2n}=K_n' \;\alpha^n (1-\alpha)^n$ for
some positive integers $K_n$ and $K_n'$, and that $\ga_{n}$ is a polynomial in $\alpha$ of order $2n$.
We have that for $|z|<1$, the three power  series in $z$, $\Pa(z), \Qa(z)$ and $\sum_{n=0}^\infty \ga_{n} \;z^{n}$ converge absolutely. 
Furthermore, since
\begin{align*}
\Da \pa_{2n}&=K_n \Da (\alpha^n (1-\alpha)^n)=\pa_{2n} (\frac{n}{\alpha}-\frac{n}{1-\alpha}),\\
 \Da \qa_{2n}&=K_n' \Da (\alpha^n (1-\alpha)^n)=\qa_{2n} (\frac{n}{\alpha}-\frac{n}{1-\alpha}),
\end{align*}
$\sum_{n=1}^\infty \Da \pa_{2n}\;z^{n}$ and $\sum_{n=0}^\infty \Da \qa_{2n}\;z^{n}$ converge absolutely
for $|z|<1$. 
It is readily shown that for $|z|<1$, 
\begin{align*}
\Da \Pa(z)=\sum_{n=1}^\infty \Da \pa_{2n}\;z^{n}\;\;\text{and}\;\;
\Da \Qa(z)=\sum_{n=0}^\infty \Da \qa_{2n}\;z^{n}.
\end{align*}
\end{remark}

To prove (\ref{eqd=2}), it suffices to show
\begin{align}\label{eqd=2.1}
\sum_{k=1}^n \Da \pa_{2k}\leq 0\;\;\text{for}\;\; \alpha \in [0,\half]\;\;\text{and}\;\; 
n\geq 1,
\end{align}
for which we need the following lemmas (whose proofs are given
at the end of this section).
\begin{lemma}\label{Lem2}
For $n\geq 1$ and $\alpha<\half$, we have  $\Da \qa_{2n}< 0$.
\end{lemma}
\begin{lemma}\label{Lem3}
For $n\geq 0$ and $\alpha \in [0,1]$, we have $\ga_{n}\geq \ga_{n+1}$.
\end{lemma}
By (\ref{eq14}), 
\begin{align}
\sum_{n=1}^\infty (\sum_{k=1}^n \Da \pa_{2k}) z^{n}&=(\sum_{n=1}^\infty \Da \pa_{2n}\; z^{n})(\sum_{n=0}^\infty
z^{n})\label{eq15}\\
&=(\sum_{n=1}^\infty \Da \pa_{2n}\; z^{n})/(1-z)\notag\\
&=\frac{\sum_{n=1}^\infty \Da \qa_{2n}\; z^{n}}{(1-z) \sum_{n=0}^\infty \ga_{n}\;z^{n}}
\notag\\
&=\frac{\sum_{n=1}^\infty \Da \qa_{2n}\; z^{n}}{1-\sum_{n=1}^\infty (\ga_{n-1}-\ga_{n}) z^{n}}\notag\\
&=\big[\sum_{n=1}^\infty \Da \qa_{2n}\; z^{n}\big] \;\big[\sum_{k=0}^\infty 
\big(Z^\alpha(z)\big)^k\big],\label{eq16}
\end{align}
where $Z^\alpha(z)=\sum_{n=1}^\infty (\ga_{n-1}-\ga_{n}) z^{n}$. Note that by Lemma \ref{Lem3},
for $|z|<1$, 
\begin{align*}
|Z^\alpha(z)|\leq \sum_{n=1}^\infty (\ga_{n-1}-\ga_n) |z|^n\leq \sum_{n=1}^\infty
(\ga_{n-1}-\ga_n) |z|=|z|<1.
\end{align*}

The left-hand side of (\ref{eq15}) is a power series in $z$ with coefficients
$\sum_{k=1}^n \Da \pa_{2k}$, $n\geq 1$. By Lemma \ref{Lem2}, for $\alpha \in [0,\half]$,
 the power series 
$\sum_{n=1}^\infty \Da \qa_{2n}\; z^{n}$ (the first term on the right-hand side of
(\ref{eq16})) has non-positive coefficients. By Lemma \ref{Lem3}, 
for $\alpha \in [0,1]$,
the power series $Z^\alpha(z)=\sum_{n=1}^\infty (\ga_{n-1}-\ga_{n}) z^{n}$ has non-negative coefficients, so that
 $\sum_{k=0}^\infty \big(Z^\alpha(z)\big)^k$ (the second term on the right-hand side of
(\ref{eq16})) is also a power series in $z$ with non-negative
 coefficients. It follows that the right-hand side of (\ref{eq16}) is a power series
 in $z$ with non-positive coefficients. Hence by (\ref{eq15}) and (\ref{eq16}), 
 the coefficients of the power series 
 $\sum_{n=1}^\infty (\sum_{k=1}^n \Da \pa_{2k}) z^{n}$ are non-positive for 
 $\alpha \in [0,\half]$, establishing
 (\ref{eqd=2.1}) and completing the proof of Theorem \ref{Thm1} for $d=2$. It remains to prove Lemmas \ref{Lem2} and \ref{Lem3}.
 
\begin{proof}[Proof of Lemma \ref{Lem2}]
We prove the lemma by induction on $n$. For $n=1$,
\begin{align}\label{eq16.1}
\Da \qa_2=\Da (\half \alpha^2+\half (1-\alpha)^2)=(2\alpha-1)< 0\;\;\text{for}\;\; \alpha<\half.
\end{align}
For the induction step, we need the following result (to be proved later)
\begin{align}\label{ID}
\Da \fa_{2n}=-2\alpha (1-\alpha) \Da \qa_{2n}.
\end{align}
Suppose the lemma holds for $n=m\geq 1$, i.e.
$\Da \qa_{2m}< 0$. 
By Lemma \ref{Lem1} and (\ref{eq16.1})--(\ref{ID}),
\begin{align*}
\Da (\frac{2m+2}{2m+1}\; \qa_{2m+2})&=\Da (2\;\qa_2\;\qa_{2m}+ \fa_{2m})\\
&=2 (\Da \qa_2) \qa_{2m}+ 2 (\half \alpha^2+\half (1-\alpha)^2) \Da \qa_{2m}
-2 \alpha (1-\alpha) \Da \qa_{2m}\\
&=2 (\Da \qa_2) \qa_{2m}+(2\alpha-1)^2 \;\Da \qa_{2m}< 0,
\end{align*}
completing the induction step. 

It remains to prove (\ref{ID}).
 For $\mathbf{h}=(\alpha, 1-\alpha)$ and $d=2$, 
\begin{align}
\MP(\Sh_{2n}&=\bfe_d^{1}+\bfe_d^2)
= \MP(\Sa_{2n}=(1,1))\notag\\
&=\sum_{k=1}^n \frac{(2n)!}{(k-1)!\; k!\; (n-k)!\; (n-k+1)!} (\half \alpha)^{2k-1}
(\half (1-\alpha))^{2n-2k+1}. \label{eq16.11}
\end{align}
By (\ref{eq3}),
\begin{align*}
\fa_{2n}&(=\fh_{2n})=2 \alpha (1-\alpha) \MP(\Sa_{2n}=(1,1))\\
&=(\half)^{2n-1} \sum_{k=1}^n \frac{(2n)!}{(k-1)!\; k!\; (n-k)!\; (n-k+1)!} \alpha^{2k}
(1-\alpha)^{2n-2k+2}.
\end{align*}
Then
\begin{align}
\Da \fa_{2n}&=(\half)^{2n-1} \Big(\sum_{k=1}^n \frac{2(2n)!}{[(k-1)!]^2 (n-k)!\;(n-k+1)!} \alpha^{2k-1}
(1-\alpha)^{2n-2k+2}\notag\\
&\qquad \qquad \qquad-\sum_{k=1}^n \frac{2(2n)!}{(k-1)!\;k!\; [(n-k)!]^2} \alpha^{2k}
(1-\alpha)^{2n-2k+1}\Big)\notag\\
&=2 \alpha (1-\alpha) (\half)^{2n} \Big(\sum_{k=1}^n \frac{2(2n)!}{[(k-1)!]^2 (n-k)!\;(n-k+1)!} \alpha^{2k-2} (1-\alpha)^{2n-2k+1}\notag\\
&\qquad \qquad \qquad \qquad \qquad -\sum_{k=1}^n \frac{2(2n)!}{(k-1)!\;k!\; [(n-k)!]^2} \alpha^{2k-1}
(1-\alpha)^{2n-2k}\Big)\;.\label{eq16.2}
\end{align}
On the other hand,
\begin{align}
\Da \qa_{2n}&=(\half)^{2n} \Da\Big(\frac{(2n)!}{(n!)^2} (1-\alpha)^{2n}+\sum_{0< k<n} \frac{(2n)!}{(k!)^2 [(n-k)!]^2} \alpha^{2k} 
(1-\alpha)^{2n-2k} +\frac{(2n)!}{(n!)^2} \alpha^{2n}\Big)\notag\\
&=(\half)^{2n} \Big(-\frac{2(2n)!}{(n-1)!\; n!} (1-\alpha)^{2n-1}+\sum_{0< k<n} \frac{2 (2n)!}{(k-1)!\; k! \;[(n-k)!]^2}\alpha^{2k-1} (1-\alpha)^{2n-2k}\notag\\
&\qquad \qquad \qquad -\sum_{0< k <n}\frac{2(2n)!}{(k!)^2 (n-k)!\; (n-k-1)!} \alpha^{2k} (1-\alpha)^{2n-2k-1}
+ \frac{2(2n)!}{(n-1)!\; n!} \alpha^{2n-1}\Big)\notag\\
&=(\half)^{2n}\Big(-\sum_{k=0}^{n-1} \frac{2(2n)!}{(k!)^2 (n-k)!\; (n-k-1)!} 
\alpha^{2k} (1-\alpha)^{2n-2k-1}\notag\\
&\qquad \qquad \qquad+ \sum_{k=1}^n \frac{2(2n)!}{(k-1)!\;k!\; [(n-k)!]^2} 
\alpha^{2k-1} (1-\alpha)^{2n-2k} \Big). \label{eq16.3}
\end{align}
Now (\ref{ID}) follows from (\ref{eq16.2}) and (\ref{eq16.3}).
The proof is complete.
\end{proof}

Next we prove Lemma \ref{Lem3}. As a part of the long proof, we first 
establish some auxiliary results as stated in Lemmas \ref{Lem4} and \ref{Lem5} below.
\begin{lemma}\label{Lem4}
(i) For $\alpha=0, 1$, $\MP(\Sa_{2n}=(1,1))=0$ for all $n$.

(ii) For $\alpha \in (0,1)$ and 
$n>m\geq 0$, 
\begin{align*}
\frac{\MP(\Sa_{2n}=(1,1))}{\qa_{2n}}\geq 
\frac{\MP(\Sa_{2m}=(1,1))}{\qa_{2m}}\;.
\end{align*}

(iii)
For $\alpha \in [0,1]$ and
 $n>m\geq 0$, $\fa_{2n}/\qa_{2n} \geq \fa_{2m}/\qa_{2m}$.
The inequality is an equality for $\alpha=0, 1$.
\end{lemma}
\begin{proof}[Proof of Lemma \ref{Lem4}]
(i) For $\alpha=0, 1$, the random walk $(\Sa_n)_{n\geq 0}$ is  1-dimensional so that it can
never visit the point $(1,1)$. It follows that $\MP(\Sa_{2n}=(1,1))=0$ for all $n$.

(ii) 
For $\alpha \in (0,1)$, let
\begin{align}\label{Lem4.1}
a_{n,k}\;(=a_{n,k}^\alpha)=\binom{n}{k} \alpha^k  (1-\alpha)^{n-k}\;\;\text{for}\;\;0\leq k\leq n,
\end{align}
and $a_{n,k}=0$ for $k<0$ or $k>n$. Then
\begin{align}
\qa_{2n}&=\sum_{k=0}^n \frac{(2n)!}{[k!\;(n-k)!]^2} (\half \alpha)^{2k} (\half (1-\alpha))^{2n-2k}
\notag\\
&=(\half)^{2n} \binom{2n}{n} \sum_{k=0}^{n} \binom{n}{k}^2 \alpha^{2k} (1-\alpha)^{2n-2k}\notag\\
&=(\half)^{2n} \binom{2n}{n} \sum_{k=0}^{n} (a_{n,k})^2.\label{Lem4.2}
\end{align}
By (\ref{eq16.11}),
\begin{align}
\MP(\Sa_{2n}=(1,1))&=(\half)^{2n} \binom{2n}{n} 
\sum_{k=1}^{n} \binom{n}{k-1} \binom{n}{k} 
\alpha^{2k-1} (1-\alpha)^{2n-2k+1}\notag\\
&=(\half)^{2n} \binom{2n}{n}  \sum_{k=1}^{n} a_{n,k-1}\; a_{n,k}\notag\\
&=(\half)^{2n} \binom{2n}{n}  \sum_{k=0}^{n-1} a_{n,k}\; a_{n,k+1}.\label{Lem4.3}
\end{align}
Let $A_n=\sum_{k=0}^n (a_{n,k})^2, B_n=\sum_{k=0}^{n-1} a_{n,k}\;a_{n,k+1}$ and
$C_n=\sum_{k=0}^{n-2} a_{n,k} \; a_{n,k+2}$.  Then 
\begin{align*}
\qa_{2n}=(\half)^{2n} \binom{2n}{n} A_n\;\;\text{and}\;\; \MP(\Sa_{2n}=(1,1))
=(\half)^{2n} \binom{2n}{n}  B_n,
\end{align*}
so that
\begin{align*}
\frac{\MP(\Sa_{2(n+1)}=(1,1))}{\qa_{2(n+1)}}
-\frac{\MP(\Sa_{2n}=(1,1))}{\qa_{2n}}=
\frac{A_n B_{n+1}-A_{n+1}B_n}{A_n A_{n+1}}.
\end{align*}
We claim that
\begin{align}\label{eq17}
A_n B_{n+1}- A_{n+1}B_n\geq 0,
\end{align}
which implies
\begin{align*}
\frac{\MP(\Sa_{2(n+1)}=(1,1))}{\qa_{2(n+1)}}
\geq \frac{\MP(\Sa_{2n}=(1,1))}{\qa_{2n}},
\end{align*}
proving (ii). 
It remains to prove (\ref{eq17}).
We have 
\begin{align*}
a_{n+1,k}&=\binom{n+1}{k} \alpha^k (1-\alpha)^{n+1-k}
=\Big(\binom{n}{k}+\binom{n}{k-1}\Big) 
\alpha^k (1-\alpha)^{n-k+1}\\
&=(1-\alpha) a_{n,k}+\alpha a_{n,k-1},
\end{align*}
and 
\begin{align*}
a_{n+1,k+1}=(1-\alpha) a_{n,k+1}+\alpha a_{n,k}.
\end{align*}
So
\begin{align*}
A_{n+1}&=\sum_{k=0}^{n+1} (a_{n+1,k})^2=\sum_{k=0}^{n+1} ((1-\alpha) a_{n,k}+\alpha\; a_{n,k-1})^2\\
&=(1-\alpha)^2 \sum_{k=0}^n (a_{n,k})^2 +\alpha^2 \sum_{k=1}^{n+1} (a_{n,k-1})^2
+2\alpha(1-\alpha) \sum_{k=1}^{n}
a_{n,k}\; a_{n,k-1}\\
&=(\alpha^2+(1-\alpha)^2) A_n+2\alpha(1-\alpha) B_n,\\
B_{n+1}&=\sum_{k=0}^n a_{n+1,k}\;a_{n+1,k+1}\\
&=\sum_{k=0}^n ((1-\alpha)a_{n,k}+\alpha\;a_{n,k-1})((1-\alpha)a_{n,k+1}+\alpha\; a_{n,k})\\
&=(\alpha^2+(1-\alpha)^2) B_n+\alpha (1-\alpha)(A_n+C_n).
\end{align*}
Next we show 
\begin{align}\label{eq18}
A_n(A_n +C_n) \geq 2 B_n^2,
\end{align}
which implies (\ref{eq17}).
We have
\begin{align*}
A_n(A_n+C_n)&=\sum_{k=0}^n (a_{n,k})^2 \Big[\sum_{k=0}^n (a_{n,k})^2+\sum_{k=0}^{n-2} a_{n,k}\; a_{n,k+2}\Big]\\
&=\half \sum_{k=0}^n (a_{n,k})^2 \Big[ (a_{n,0})^2+(a_{n,1})^2+\sum_{k=0}^{n-2}
(a_{n,k}+a_{n,k+2})^2+(a_{n,n-1})^2+(a_{n,n})^2\Big]\\
&\geq \half \sum_{k=0}^n (a_{n,k})^2 \Big[ (a_{n,1})^2+\sum_{k=1}^{n-1}
(a_{n,k-1}+a_{n,k+1})^2+(a_{n,n-1})^2\Big]\\
&\geq \half \Big[a_{n,0}\;a_{n,1}+\sum_{k=1}^{n-1}a_{n,k}(a_{n,k-1}+a_{n,k+1})+a_{n,n}\;a_{n,n-1}\Big]^2
\\
&=\half (2 \sum_{k=0}^{n-1} a_{n,k}\;a_{n,k+1})^2=2B_n^2,
\end{align*}
where the second inequality is an application of
 Cauchy-Schwarz inequality. This proves (\ref{eq18}).
 
 (iii) Since $\fa_{2n}=2\alpha (1-\alpha) \MP(\Sa_{2n}=(1,1))$,
 part (iii) follows from parts (i) and (ii).
\end{proof}
\begin{lemma}\label{Lem5}
\begin{align*}
(i)&\;\; \text{For}\; \alpha \in [0,1]\;\text{and}\; n\geq 0,\\
&\qquad \qquad \qa_{2n}=\MP(\Sa_{2n}=(0,0))\geq
\MP(\Sa_{2n}=(1,1)).\\
(ii)&\;\;\text{For}\; \alpha \in [0,1]\;\text{and}\;n\geq 0,\\ 
&\qquad \qquad \qa_{2n}-\qa_{2n+2}\geq \frac{1}{2n+1}\; \qa_{2n+2}.\\
&\;\;\text{The inequality is an equality for}\; \alpha=0, 1.\\
(iii)&\;\;\text{For}\;\alpha \in [0,1]\;\text{and}\;0\leq k \leq \frac{n}{2}-1,\\
&\qquad \qquad \frac{\qa_{2k}\; \qa_{2n-2k}}{\qa_{2k+2}\;\qa_{2n-2k-2}} \geq \frac{2k+2}{2k+1}\;\frac{2n-2k-1}{2n-2k}.
\\
& \;\;\text{The inequality} \;\text{is an equality for}\;\alpha=0, 1.
\\
(iv)& \;\; \text{For}\;\alpha \in [0,1],\;(\text{even})\; n=2m\geq 2\;\text{and}\; 0\leq k<m, \\
&\qquad \qquad \frac{\qa_{2k}\;\qa_{4m-2k}}{(\qa_{2m})^2}\geq \frac{\binom{2k}{k}\binom{4m-2k}{2m-k}}{\binom{2m}{m}^2}.\\
&\;\;\text{The inequality is an equality for}\; \alpha=0, 1.\\
(v)& \;\; \text{For}\;\alpha \in[0,1],\;(\text{odd})\; n=2m+1\geq 3\;
\text{and}\; 0\leq k<m,\\
&\qquad \qquad \frac{\qa_{2k}\;\qa_{4m-2k+2}}{\qa_{2m}\;\qa_{2m+2}}\geq 
\frac{\binom{2k}{k}\binom{4m-2k+1}{2m-k}}{\binom{2m}{m}\;\binom{2m+1}{m}}.\\
&\;\; \text{The inequality is an equality for}\;
\alpha=0, 1.
\end{align*}
\end{lemma}
\begin{proof}[Proof of Lemma \ref{Lem5}]
(i) For $\alpha=0,1,\; \MP(\Sa_{2n}=(1,1))=0<\MP(\Sa_{2n}=(0,0))$.
For $\alpha \in (0,1)$, by (\ref{Lem4.2}) and (\ref{Lem4.3}), 
\begin{align*}
\MP(\Sa_{2n}=(0,0))&- \MP(\Sa_{2n}=(1,1))\\
&=(\half)^{2n} \binom{2n}{n} \Big[
\sum_{k=0}^n (a_{n,k})^2-\sum_{k=0}^{n-1} a_{n,k}\;a_{n,k+1}\Big]\\
&\geq (\half)^{2n} \binom{2n}{n} \Big[\sum_{k=0}^n (a_{n,k})^2-\sum_{k=0}^{n-1}(\half)
((a_{n,k})^2+(a_{n,k+1})^2)\Big]\\
&\geq 0,
\end{align*}
where the first inequality is an application of  Cauchy-Schwarz inequality, proving (i). 

(ii) For  $\alpha=0, 1$,  we have $\qa_2=\half$ and $\fa_{2n}=2 \alpha (1-\alpha) 
\MP(\Sa_{2n}=(1,1))=0$. By Lemma \ref{Lem1},
$\frac{2n+2}{2n+1}\;\qa_{2n+2}=2 \qa_2 \;\qa_{2n}=\qa_{2n}$, proving that
the inequality in part (ii) is an equality for $\alpha=0, 1$.

For $\alpha \in (0,1)$, by part (i),
\begin{align*}
2\alpha(1-\alpha)\qa_{2n}-\fa_{2n}&=2 \alpha (1-\alpha)\big[\MP(\Sa_{2n}=(0,0))-
\MP(\Sa_{2n}=(1,1))\big]\\
&\geq 0.
\end{align*} 
By Lemma \ref{Lem1},
\begin{align*}
\frac{2n+2}{2n+1}\;\qa_{2n+2}&=2\qa_2\;\qa_{2n}+\fa_{2n}\\
&=2\qa_2\;\qa_{2n}+2\alpha(1-\alpha) \qa_{2n}-(2\alpha(1-\alpha)\qa_{2n}-\fa_{2n})\\
&\leq (2\qa_2+2\alpha(1-\alpha))\qa_{2n}=\qa_{2n},
\end{align*}
proving (ii).

(iii) By Lemma \ref{Lem1}, 
\begin{align}
\frac{2k+2}{2k+1}\; \qa_{2k+2}\;\qa_{2n-2k-2}&=(2\qa_2\;\qa_{2k}+\fa_{2k}) \qa_{2n-2k-2},\label{eq20}\\
\frac{2n-2k}{2n-2k-1}\; \qa_{2n-2k} \;\qa_{2k}&=(2\qa_2\;\qa_{2n-2k-2}+\fa_{2n-2k-2})\qa_{2k}.\label{eq21}
\end{align}
 Since $2k\leq 2n-2k-2$, we have $\fa_{2n-2k-2}\;\qa_{2k} \geq \fa_{2k} \;\qa_{2n-2k-2}$ by Lemma \ref{Lem4}(iii),
implying that the right-hand side of (\ref{eq20}) is less than or equal to that of (\ref{eq21}).
Consequently, the left-hand side of (\ref{eq21}) is greater than or equal to that of (\ref{eq21}),
which is equivalent to the inequality in (iii).

 For $\alpha=0, 1$, by Lemma \ref{Lem4}(iii), the right-hand sides of (\ref{eq20}) and (\ref{eq21}) are
equal, implying that the inequality in (iii) is an equality.

(iv) Since
\begin{align}
\prod_{j=k}^{m-1} (2j+1)&=\frac{(2m)!}{(2k)! \prod_{j=k}^{m-1} (2j+2)}=
\frac{(2m)!}{(2k)!\; 2^{m-k} m!/k!}\;\;\;\text{and}\;\label{eq20.3}\\
\prod_{j=k}^{m-1} (4m-2j-1)&=\frac{(4m-2k)!}{(2m)! \prod_{j=k}^{m-1} (4m-2j)}
=\frac{(4m-2k)!}{(2m)!\; 2^{m-k} (2m-k)!/m!},\notag
\end{align}
we have by (iii),
\begin{align}
\frac{\qa_{2k}\;\qa_{4m-2k}}{(\qa_{2m})^2}
&=\prod_{j=k}^{m-1} \frac{\qa_{2j}\; \qa_{4m-2j}}{\qa_{2j+2}\;\qa_{4m-2j-2}} \notag\\
&\geq \prod_{j=k}^{m-1} \frac{2j+2}{2j+1}\;\frac{4m-2j-1}{4m-2j}\label{eq20.4}\\
&=\prod_{j=k}^{m-1} \frac{4m-2j-1}{2j+1}\;\prod_{j=k}^{m-1} \frac{2(j+1)}{2(2m-j)}\notag\\
&=\Big[\frac{(4m-2k)!\;(2k)!\; (m!)^2}{((2m)!)^2\;(2m-k)!\;k!}\Big] \Big[\frac{m!/k!}{ (2m-k)!/m!}\Big]
\notag\\
&=\frac{\binom{2k}{k}\binom{4m-2k}{2m-k}}{\binom{2m}{m}^2},\notag
\end{align}
establishing the inequality in (iv). For $\alpha=0, 1$, by (iii), the inequality in (\ref{eq20.4}) is an equality, so that the inequality in (iv) is an equality for $\alpha=0, 1$.

(v) By (\ref{eq20.3}) and
\begin{align*}
\prod_{j=k}^{m-1} (4m-2j+1)&=\frac{(4m-2k+1)!}{(2m+1)! \prod_{j=k}^{m-1} (4m-2j)}
=\frac{(4m-2k+1)!}{(2m+1)!\; 2^{m-k} (2m-k)!/m!},
\end{align*}
we have by (iii),
\begin{align}
\frac{\qa_{2k}\;\qa_{4m-2k+2}}{\qa_{2m}\;\qa_{2m+2}}
&=\prod_{j=k}^{m-1} \frac{\qa_{2j}\; \qa_{4m-2j+2}}{\qa_{2j+2}\;\qa_{4m-2j}}\notag \\
&\geq \prod_{j=k}^{m-1} \frac{2j+2}{2j+1}\;\frac{4m-2j+1}{4m-2j+2}\label{eq20.5}\\
&=\prod_{j=k}^{m-1} \frac{4m-2j+1}{2j+1}\;\prod_{j=k}^{m-1} \frac{2(j+1)}{2(2m-j+1)}\notag\\
&=\Big[\frac{(4m-2k+1)!\;(2k)!\; (m!)^2}{(2m)!\;(2m+1)!\;(2m-k)!\;k!}\Big] \Big[\frac{m!/k!}
{ (2m-k+1)!/(m+1)!}\Big]\notag\\
&=\frac{\binom{2k}{k}\binom{4m-2k+1}{2m-k}}{\binom{2m}{m}\;\binom{2m+1}{m}},\notag
\end{align}
establishing the inequality in (v). For $\alpha=0, 1$, by (iii), the inequality in (\ref{eq20.5}) is an equality, so that the inequality in (v) is an equality for $\alpha=0, 1$.
The proof of Lemma \ref{Lem5} is complete.
\end{proof}

We are now ready to prove Lemma \ref{Lem3}.

\begin{proof}[Proof of Lemma \ref{Lem3}]
Recall the well-known  identity for convolution of  central binomial coefficients (see
\emph{e.g.} \cite[(5.39), p. 187]{GKP}),
\begin{align}\label{eq20.6}
\sum_{k=0}^n \binom{2k}{k}\;\binom{2n-2k}{n-k}=4^n.
\end{align}
For $\alpha=0,1$, 
\begin{align*}
\ga_n&=\sum_{k=0}^n \qa_{2k}\;\qa_{2n-2k}= \sum_{k=0}^n \Big[(\half)^{2k} \binom{2k}{k}\Big] 
\Big[(\half)^{2n-2k} \binom{2n-2k}{n-k}\Big]\\
&=(\half)^{2n} \sum_{k=0}^n \binom{2k}{k}\;\binom{2n-2k}{n-k}=1,
\end{align*}
by (\ref{eq20.6}). So 
\begin{align}\label{eq20.60}
\ga_{n-1}-\ga_n=0\;\;\text{for}\;\; \alpha=0,1\;\;\text{and} \;\;n\geq 1.
\end{align} 
For 
$\alpha \in [0,1]$, $\ga_1=\sum_{k=0}^1 \qa_{2k}\;\qa_{2-2k}=2\;\qa_0\;\qa_2=\alpha^2+(1-\alpha)^2
\leq 1=\ga_0$. To show $\ga_{n-1}-\ga_n \geq 0$ for $n\geq 2$, we consider the two cases
even $n=2m\geq 2$ and odd $n=2m+1\geq 3$ separately.

Case  $n=2m\geq 2$:  By parts (ii) and (iv) of Lemma \ref{Lem5}, we have for $\alpha \in [0,1]$,
\begin{align}
\ga_{2m-1}-\ga_{2m}&=\sum_{k=0}^{2m-1} \qa_{2k}\;\qa_{4m-2k-2}-\sum_{k=0}^{2m} \qa_{2k}\;\qa_{4m-2k}
\notag\\
&=2\sum_{k=0}^{m-1}  \qa_{2k}\;\qa_{4m-2k-2}- 2 \sum_{k=0}^{m-1} \qa_{2k}\;\qa_{4m-2k}-(\qa_{2m})^2
\notag\\
&=2 \sum_{k=0}^{m-1} \qa_{2k}\;(\qa_{4m-2k-2}-\qa_{4m-2k})-(\qa_{2m})^2 \notag\\
&\geq 2 \sum_{k=0}^{m-1} \qa_{2k}\;\Big(\frac{1}{4m-2k-1}\; \qa_{4m-2k}\Big)
-(\qa_{2m})^2 \label{eq20.7}\\
&\geq 2 \sum_{k=0}^{m-1} \Big(\frac{1}{4m-2k-1}\Big) \frac{\binom{2k}{k} \binom{4m-2k}{2m-k}}
{\binom{2m}{m}^2}\;(\qa_{2m})^2 -  (\qa_{2m})^2 \label{eq20.8}\\
&=\Big(2 \sum_{k=0}^{m-1}  \frac{\binom{2k}{k} \binom{4m-2k}{2m-k}}
{(4m-2k-1) \binom{2m}{m}^2} -1 \Big) \;(\qa_{2m})^2.\notag
\end{align}
For $\alpha=0,1$, by parts (ii) and (iv) of Lemma \ref{Lem5} again, the inequalities in
(\ref{eq20.7}) and (\ref{eq20.8}) are equalities. In other words, for
$\alpha \in (0,1)$, 
\begin{align}\label{eq20.11}
\ga_{2m-1}-\ga_{2m}\geq \Big(2 \sum_{k=0}^{m-1}  \frac{\binom{2k}{k} \binom{4m-2k}{2m-k}}
{(4m-2k-1) \binom{2m}{m}^2} -1 \Big) \;(\qa_{2m})^2,
\end{align}
and for $\alpha=0,1$, 
\begin{align}\label{eq20.12}
\ga_{2m-1}-\ga_{2m}=\Big(2 \sum_{k=0}^{m-1}  \frac{\binom{2k}{k} \binom{4m-2k}{2m-k}}
{(4m-2k-1) \binom{2m}{m}^2} -1 \Big) \;(\qa_{2m})^2.
\end{align}
On the other hand, by (\ref{eq20.60}), $\ga_{2m-1}-\ga_{2m}=0$  for $\alpha=0,1$, which together with
(\ref{eq20.12}) implies
\begin{align}\label{eq20.10}
2 \sum_{k=0}^{m-1}  \frac{\binom{2k}{k} \binom{4m-2k}{2m-k}}
{(4m-2k-1) \binom{2m}{m}^2} -1=0.
\end{align}
For $\alpha \in (0,1)$, by (\ref{eq20.11}) and (\ref{eq20.10}),
\begin{align*}
\qa_{2m-1}-\qa_{2m}\geq \Big(2 \sum_{k=0}^{m-1}  \frac{\binom{2k}{k} \binom{4m-2k}{2m-k}}
{(4m-2k-1) \binom{2m}{m}^2} -1 \Big) \;(\qa_{2m})^2=0,
\end{align*}
completing the proof for the case $n=2m$.

Case $n=2m+1\geq 3$:
By parts (ii) and (v) of Lemma \ref{Lem5}, we have for $\alpha \in [0,1]$,
\begin{align}
\ga_{2m}-\ga_{2m+1}&=\sum_{k=0}^{2m} \qa_{2k}\;\qa_{4m-2k}-\sum_{k=0}^{2m+1} \qa_{2k}\;\qa_{4m+2-2k}
\notag\\
&=2\sum_{k=0}^{m-1}  \qa_{2k}\;\qa_{4m-2k}+(\qa_{2m})^2
- 2 \sum_{k=0}^{m} \qa_{2k}\;\qa_{4m-2k+2}
\notag\\
&=2 \sum_{k=0}^{m-1} \qa_{2k}\;(\qa_{4m-2k}-\qa_{4m-2k+2})+\qa_{2m} (\qa_{2m}-\qa_{2m+2})
-\qa_{2m}\;\qa_{2m+2} \notag\\
&\geq 2 \sum_{k=0}^{m-1} \qa_{2k}\;\Big(\frac{1}{4m-2k+1}\; \qa_{4m-2k+2}\Big)
+\qa_{2m} \Big(\frac{1}{2m+1}\; \qa_{2m+2}\Big)- \qa_{2m}\;\qa_{2m+2} \label{eq20.13}\\
&\geq 2 \sum_{k=0}^{m-1} \Big(\frac{1}{4m-2k+1}\Big) \frac{\binom{2k}{k} \binom{4m-2k+1}{2m-k}}
{\binom{2m}{m} \binom{2m+1}{m}}\;\qa_{2m}\;\qa_{2m+2}+\frac{1}{2m+1}\;\qa_{2m}\;\qa_{2m+2}
 -  \qa_{2m}\;\qa_{2m+2} \label{eq20.14}\\
&=\Big(2 \sum_{k=0}^{m-1}  \frac{\binom{2k}{k} \binom{4m-2k+1}{2m-k}}
{(4m-2k+1) \binom{2m}{m} \binom{2m+1}{m}} -\frac{2m}{2m+1} \Big) \;\qa_{2m}\;\qa_{2m+2}.\notag
\end{align}
For $\alpha=0,1$, by parts (ii) and (v) of Lemma \ref{Lem5} again, the inequalities in
(\ref{eq20.13}) and (\ref{eq20.14}) are equalities, which together with (\ref{eq20.60}) imply that 
 for
$\alpha=0,1$, 
\begin{align}
0&=\ga_{2m}-\ga_{2m+1}\notag\\
&=\Big(2 \sum_{k=0}^{m-1}  \frac{\binom{2k}{k} \binom{4m-2k+1}{2m-k}}
{(4m-2k+1) \binom{2m}{m} \binom{2m+1}{m}} -\frac{2m}{2m+1} \Big) \;\qa_{2m}\;\qa_{2m+2},
\label{eq20.15}
\end{align}
and for $\alpha \in (0,1)$,
\begin{align*}
\ga_{2m}-\ga_{2m+1}\geq \Big(2 \sum_{k=0}^{m-1}  \frac{\binom{2k}{k} \binom{4m-2k+1}{2m-k}}
{(4m-2k+1) \binom{2m}{m} \binom{2m+1}{m}} -\frac{2m}{2m+1} \Big) \;\qa_{2m}\;\qa_{2m+2}=0,
\end{align*}
completing the proof for the case $n=2m+1$. The proof of Lemma \ref{Lem3} is complete.
\end{proof}

\begin{remark}
The equality in (\ref{eq20.15}) yields the combinatorial identity
\begin{align}\label{eq20.16}
2 \sum_{k=0}^{m-1}  \frac{\binom{2k}{k} \binom{4m-2k+1}{2m-k}}
{(4m-2k+1) \binom{2m}{m} \binom{2m+1}{m}} -\frac{2m}{2m+1}=0\;.
\end{align}
This and  the  identity in 
(\ref{eq20.10}) will be needed in the next section.
\end{remark}

\begin{remark}\label{section3}
By Lemma \ref{Lem2}, we have $q_{2n}^{\alpha'}>q_{2n}^{\alpha''}$ for $n\geq 1$ and
$0\leq \alpha'<\alpha''\leq \half$. Equivalently, we have $q_{2n}^{\mathbf{h}'}>q_{2n}^{\mathbf{h}''}$ if 
$\mathbf{h}'=(\alpha', 1-\alpha') \succ \mathbf{h}''=(\alpha'',1-\alpha'')$.
\end{remark}

\section{The general case $d\ge 3$}

For two $d$-dimensional probability vectors $\mathbf{h}'$ and
$\mathbf{h}''$ with $d\geq 3$, suppose $\mathbf{h}' \succ \mathbf{h}''$.
Then there exists a sequence of probability vectors $(\mathbf{h}_\ell, \ell=0,\dots, k)$
such that $\mathbf{h}_0=\mathbf{h}', \mathbf{h}_k=\mathbf{h}''$, and
for $\ell=1,\dots, k$, 
$\mathbf{h}_{\ell-1} \succ \mathbf{h}_{\ell}$ and they differ only in two entries
(see \emph{e.g.} \cite[Lemma B.1, Chapter 2]{MOA}).
Hence, to show that $\Tha$ is stochastically smaller than $\Thb$, 
it suffices to assume that
$\mathbf{h}'$ and $\mathbf{h}''$ differ only in two entries. It is convenient to
introduce the representation $(\alpha, \xi)$ of a $d$-dimensional probability vector 
$\mathbf{h}=(h_1,\dots, h_d)$, where for some $1\leq i<j\leq d$,
$\xi=(i,j, h_k, k \in \{1,\dots, d\}\setminus \{i,j\})$ and
$\alpha=h_i/(h_i+h_j)$. (In this representation,
it is implicitly assumed that $h_i+h_j>0$.)
As an example, for $d=4$, $(\alpha,\xi)$ with $\alpha=0.3$ and
$\xi=(2,4, 0.2, 0.4)$ represents the vector $\mathbf{h}=(0.2, 0.12, 0.4, 0.28)$.
For two probability vectors $\mathbf{h}'$ and $\mathbf{h}''$
with respective representations $(\alpha', \xi)$ and $(\alpha'', \xi)$ (sharing the same $\xi$), 
$\mathbf{h}' \succ \mathbf{h}''$ if and only if 
$|\alpha'-\half| > |\alpha''-\half|$. As in the preceding section, we write
$\Th=\Tax, \Sh_n=\Sax_n, \Xh_n=\Xax_n, \ph_n=\pax_n, \qh_n=\qax_n, \Ph(z)=\Pax(z),
\Qh(z)=\Qax(z), \gh_n=\gax_n, \fh_n=\fax_n$. We use these two types of expressions interchangeably, depending
on whether a pair of $(i,j)$ with $1\leq i<j \leq d$ requires special attention..
We also write $\mathbf{h}\equiv (\alpha, \xi)$ to mean that
$(\alpha, \xi)$ is the representation of $\mathbf{h}$.
 (For the 2-dimensional case, necessarily
$\xi=(1,2)$, so that $\xi$ need not be included  in the representation, and 
$\alpha$ alone represents the 2-dimensional probability 
vector $(\alpha, 1-\alpha)$, i.e. $\mathbf{h}=(\alpha, 1-\alpha)\equiv \alpha$.
 In this way, it is consistent with the notation in Section 3
 where we write $\Sh_n=\Sa_n, \ph_n=\pa_n, \qh_n=\qa_n, \gh_n=\ga_n,
 \Ph(z)=\Pa(z), \Qh(z)=\Qa(z)$.)

As discussed in the preceding paragraph, to prove Theorem \ref{Thm1} for $d\geq 3$, it suffices
to show that 
\begin{align}\label{eq4.1}
T^{\alpha',\;\xi}\;\;\text{is stochastically smaller than}\;\;T^{\alpha'',\; \xi},
\end{align}
for any $\xi$ and $0\leq \alpha'<\alpha''\leq \half$.
Our proof will closely follow  the arguments in the preceding
section, while adopting a conditional approach to deal with the presence of $\xi$.
We will show that for any given $\xi$,
\begin{align}\label{eq4.2}
\sum_{k=1}^n \Da \pax_{2k}\leq 0 \;\;\text{for}\;\; \alpha \in [0,\half]\;\,\text{and}\;\; n\geq 1,
\end{align}
which implies (\ref{eq4.1}). (Note that $\xi$ is treated as fixed in $\Da \pax_{2k}$.)

We need the following lemmas, which are extensions of Lemmas \ref{Lem2} and \ref{Lem3} to
$d \geq 3$.

\begin{lemma}\label{Lem6}
For any $\xi$ and  $\alpha< \half$, we have $\Da \qax_{2n} < 0$ for $n\geq 1$.
\end{lemma}

\begin{lemma}\label{Lem7}
For any $\mathbf{h}$, we have $\gh_n \geq \gh_{n+1}$ for $n \geq 0$.
\end{lemma}

Assuming Lemmas \ref{Lem6} and \ref{Lem7} hold, (\ref{eq4.2}) 
follows by the same arguments right below
the statement of Lemma \ref{Lem3}. More precisely, similar to (\ref{eq14}), we have 
\begin{align*}
\sum_{n=1}^\infty \Da \pax_{2n}\; z^{n}&=\frac{\sum_{n=1}^\infty \Da \qax_{2n}\; z^{n}}{(\Qax(z))^2}
=\frac{\sum_{n=1}^\infty \Da \qax_{2n}\; z^{n}}{\sum_{n=0}^\infty \gax_{n}\;z^{n}},
\end{align*}
implying that
\begin{align*}
\sum_{n=1}^\infty (\sum_{k=1}^n \Da \pax_{2k})\; z^{n}&=(\sum_{n=1}^\infty \Da \pax_{2n}\; z^{n})(\sum_{n=0}^\infty z^{n})\\
&=(\sum_{n=1}^\infty \Da \pax_{2n} \; z^{n})/(1-z)\\
&=\frac{\sum_{n=1}^\infty \Da \qax_{2n}\; z^{n}}{(1-z) \sum_{n=0}^\infty \gax_{n}\;z^{n}}\\
&=\frac{\sum_{n=1}^\infty \Da \qa_{2n}\; z^{n}}{1-\sum_{n=1}^\infty (\gax_{n-1}-\gax_{n}) z^{n}}\\
&=\big[\sum_{n=1}^\infty \Da \qax_{2n}\; z^{n}\big] \;\big[\sum_{k=0}^\infty 
\big(Z^{\alpha,\xi}(z)\big)^k\big],
\end{align*}
where $Z^{\alpha,\xi} (z)=\sum_{n=1}^\infty (\gax_{n-1}-\gax_{n}) z^{n}$.
The coefficients of the power series 
$\sum_{n=1}^\infty \Da \qax_{2n}\; z^{n}$ are non-positive for $\alpha \in [0, \half]$ by 
Lemma \ref{Lem6}, 
while those of
the power series $Z^{\alpha,\xi} (z)$ are non-negative by Lemma \ref{Lem7}. It follows that
the coefficients of the power series $\sum_{n=1}^\infty (\sum_{k=1}^n \Da \pax_{2k})\; z^{n}$
for $\alpha \in [0, \half]$ are non-positive,
proving (\ref{eq4.2}). It remains to prove Lemmas \ref{Lem6} and \ref{Lem7}.

\begin{proof}[Proof of Lemma \ref{Lem6}]
For $\mathbf{h}\equiv (\alpha, \xi)$, to show $\Da \qax_{2n}< 0$ for 
$\alpha <\half$ and $n \geq 1$,
we assume without loss of generality that $\xi=(1,2, h_3, \dots, h_d)$ and $0<h_3+\cdots+h_d<1$.
Let $\gamma=h_3+\cdots+h_d$ and $\mathbf{h}'=\gamma^{-1}(h_3,\dots, h_d)$,
a ($d-2$)-dimensional probability vector. Let $N_n^i=|\{1\leq k\leq n:
\Xax_k \in \{\bfe_{d}^{i}, -\bfe_d^{i}\}\}|$, the number of times the random walk
$(\Sax_0, \Sax_1, \dots, \Sax_n)$ moves in the $i$-th dimension. Then $N_n^1+N_n^2$ is binomial
with parameters $n$ and $1-\gamma$. Given $N_{n}^1+N_{n}^2=\ell$, 
the set of the first two components of $\Sax_n$ is conditionally independent of
the set of the other components of $\Sax_n$. Consequently,
\begin{align}\label{eq4.3}
\MP(\Sax_{n}=\zero_d\;\Big|\; N_{n}^1+N_{n}^2=\ell)=\MP(\Sa_\ell=(0,0))\;
\MP(\Sha_{n-\ell}=\zero_{d-2}).
\end{align}
Note that $\MP(\Sa_\ell=(0,0))=0$ for odd $\ell$. By (\ref{eq4.3}),
\begin{align}
\qax_{2n}&=\MP(\Sax_{2n}=\zero_d)\notag\\
&=\sum_{m=0}^n \binom{2n}{2m} (1-\gamma)^{2m} \gamma^{2n-2m}\;
\MP(\Sax_{2n}=\zero_d\;\Big|\; N_{2n}^1+N_{2n}^2=2m)\notag\\
&=\sum_{m=0}^n \binom{2n}{2m} (1-\gamma)^{2m} \gamma^{2n-2m}\;
\MP(\Sa_{2m}=(0,0))\;\MP(\Sha_{2n-2m}=\zero_{d-2})\notag\\
&=\sum_{m=0}^n \binom{2n}{2m} (1-\gamma)^{2m} \gamma^{2n-2m}\;
\qa_{2m}\;\MP(\Sha_{2n-2m}=\zero_{d-2}).\label{eq4.4}
\end{align}
Since by Lemma \ref{Lem2}, $\Da \qa_{2m}< 0$ for $\alpha <\half$ and $m\geq 1$, it follows from (\ref{eq4.4}) that
$\Da \qax_{2n}< 0$ for  $\alpha <\half$ and $n\geq 1$.
\end{proof}

To prove Lemma \ref{Lem7}, we  need Lemmas \ref{Lem8} and \ref{Lem9} below, which are extensions of
Lemmas \ref{Lem4} and \ref{Lem5} to $d \geq 3$.

\begin{lemma}\label{Lem8}
(i) For  $1\leq i<j\leq d$ and 
$n>m\geq 0$, 
\begin{align*}
\frac{\MP(\Sh_{2n}=\bfe_d^{i}+\bfe_d^j)}{\qh_{2n}}\geq 
\frac{\MP(\Sh_{2m}=\bfe_d^{i}+\bfe_d^j)}{\qh_{2m}}\;.
\end{align*}

(ii)
For  $n>m\geq 0$, $\fh_{2n}/\qh_{2n} \geq \fh_{2m}/\qh_{2m}$.
\end{lemma}
\begin{proof}[Proof of Lemma \ref{Lem8}]
(i) It suffices to consider $(i,j)=(1,2)$ and show that 
\begin{align}\label{eq4.5}
\frac{\MP(\Sh_{2(n+1)}=\bfe_d^{1}+\bfe_d^2)}{\qh_{2(n+1)}}\geq 
\frac{\MP(\Sh_{2n}=\bfe_d^{1}+\bfe_d^2)}{\qh_{2n}},
\end{align}
for $n\geq 0$ 
and all $d$-dimensional probability vector
$\mathbf{h}$.
We prove (\ref{eq4.5}) by induction on $d \geq 2$. The case $d=2$ follows from Lemma \ref{Lem4}(ii).
Suppose (\ref{eq4.5}) holds for some $d\geq 2$. We need to show
\begin{align}\label{eq4.6}
\frac{\MP(\Sh_{2(n+1)}=\bfe_{d+1}^{1}+\bfe_{d+1}^2)}{\qh_{2(n+1)}}\geq 
\frac{\MP(\Sh_{2n}=\bfe_{d+1}^{1}+\bfe_{d+1}^2)}{\qh_{2n}},
\end{align}
for $n\geq 0$ and all
$(d+1)$-dimensional probability vector $\mathbf{h}=(h_1,\dots, h_{d+1})$.
Let $\mathbf{h}'=(h_1,\dots,h_d)/(1-h_{d+1})$, a $d$-dimensional probability vector.
Let $N_{2n}^{d+1}$ denote the number of times the random walk $(\Sh_0, \Sh_1,\dots, \Sh_{2n})$
moves in the $(d+1)$-th dimension, which is binomial with parameters $2n$ and $h_{d+1}$. 
Then given $N_{2n}^{d+1}=\ell$, the set of the first $d$ components of $\Sh_{2n}$ is
conditionally independent of the $(d+1)$-th component of $\Sh_{2n}$. 
For the the $(d+1)$-th component of $\Sh_{2n}$ to be 0, it is necessary for $N_{2n}^{d+1}$ to be
an even number. Moreover, given $N_{2n}^{d+1}=2k$, the $(d+1)$-th component of $\Sh_{2n}$ 
equals 0 with (conditional) probability $2^{-2k} \binom{2k}{k}$. It follows that
\begin{align}
\MP(\Sh_{2n}=\bfe_{d+1}^{1}+\bfe_{d+1}^2)&=
\sum_{k=0}^n \binom{2n}{2k} h_{d+1}^{2k} (1-h_{d+1})^{2n-2k} 2^{-2k} \binom{2k}{k} 
\MP(\Sha_{2n-2k}=\bfe_d^{1}+\bfe_d^2)\notag\\
&=\sum_{k=0}^n \frac{(2n)!}{2^{2k} (k!)^2 (2n-2k)!} h_{d+1}^{2k} (1-h_{d+1})^{2n-2k}
\MP(\Sha_{2n-2k}=\bfe_d^{1}+\bfe_d^2)\notag\\
&=\sum_{k=0}^n b_{n,k} \;\MP(\Sha_{2n-2k}=\bfe_d^{1}+\bfe_d^2),\label{eq4.7}
\end{align}
where for $k=0,\dots, n$,
\begin{align}\label{eq4.70}
b_{n,k}=\frac{(2n)!}{2^{2k} (k!)^2 (2n-2k)!} h_{d+1}^{2k} (1-h_{d+1})^{2n-2k}.
\end{align}
 Similarly,
\begin{align}
\MP(\Sh_{2n}=\zero_{d+1})&=
\sum_{k=0}^n \binom{2n}{2k} h_{d+1}^{2k} (1-h_{d+1})^{2n-2k} 2^{-2k} \binom{2k}{k} 
\MP(\Sha_{2n-2k}=\zero_d)\notag\\
&=\sum_{k=0}^n b_{n,k} \;\MP(\Sha_{2n-2k}=\zero_d)=\sum_{k=0}^n b_{n,n-k} \;\qha_{2k}.\label{eq4.701}
\end{align}
We re-write (\ref{eq4.7}) as
\begin{align}
\MP(\Sh_{2n}=\bfe_{d+1}^{1}+\bfe_{d+1}^2)
&=\sum_{k=0}^n b_{n,k} \;\qha_{2n-2k} \;\frac{\MP(\Sha_{2n-2k}=\bfe_d^{1}+\bfe_d^2)}{\qha_{2n-2k}}
\notag\\
&=\sum_{k=0}^n b_{n,n-k} \;\qha_{2k} \;\frac{\MP(\Sha_{2k}=\bfe_d^{1}+\bfe_d^2)}{\qha_{2k}}.\label{eq4.71}
\end{align}
By (\ref{eq4.701})--(\ref{eq4.71}),
\begin{align}
\frac{\MP(\Sh_{2n}=\bfe_{d+1}^{1}+\bfe_{d+1}^2)}{\qh_{2n}}&=
\frac{\MP(\Sh_{2n}=\bfe_{d+1}^{1}+\bfe_{d+1}^2)}{\MP(\Sh_{2n}=\zero_{d+1})}\notag\\
&= (\sum_{k=0}^n b_{n,n-k} \;\qha_{2k})^{-1}
 \sum_{k=0}^n b_{n,n-k} \;\qha_{2k} \;\frac{\MP(\Sha_{2k}=\bfe_d^{1}+\bfe_d^2)}{\qha_{2k}},
\end{align}
which is a weighted average of $\frac{\MP(\Sha_{2k}=\bfe_d^{1}+\bfe_d^2)}{\qha_{2k}}, k=0,\dots, n$
with respective weights $w_{n,k}=b_{n,n-k}\;\qha_{2k}/\sum_{\ell=0}^n b_{n,n-\ell} \;\qha_{2\ell}$.
Similarly, $\frac{\MP(\Sh_{2(n+1)}=\bfe_{d+1}^{1}+\bfe_{d+1}^2)}{\qh_{2(n+1)}}$
is a weighted average of $\frac{\MP(\Sha_{2k}=\bfe_d^{1}+\bfe_d^2)}{\qha_{2k}}, k=0,\dots, n+1$
with respective weights 
$w_{n+1,k}=b_{n+1,n+1-k}\;\qha_{2k}/\sum_{\ell=0}^{n+1} b_{n+1,n+1-\ell} \;\qha_{2\ell}$.
By the induction hypothesis,  
\begin{align}\label{eq4.69}
\frac{\MP(\Sha_{2k}=\bfe_d^{1}+\bfe_d^2)}{\qha_{2k}}\;\;\text{is  increasing
 in}\;\; k.
 \end{align}
The ratio of the weights equals
\begin{align*}
\frac{w_{n,k}}{w_{n+1,k}}&=\frac{b_{n,n-k}\;\qha_{2k}/\sum_{\ell=0}^n b_{n,n-\ell}\; \qha_{2 \ell}}
{b_{n+1,n+1-k}\;\qha_{2k}/\sum_{\ell=0}^{n+1} b_{n+1,n+1-\ell}\; \qha_{2 \ell}}\\
&=\frac{2 (n-k+1)^2}{(n+1)(2n+1) h_{d+1}^2} \;\frac{\sum_{\ell=0}^{n+1} b_{n+1,n+1-\ell}\;\qha_{2 \ell}}
{\sum_{\ell=0}^n b_{n,n-\ell}\; \qha_{2\ell} },
\end{align*}
which is decreasing in $k=0,\dots, n$.
Setting $w_{n,n+1}=0$, we have $\sum_{k=0}^{n+1} w_{n,k}=\sum_{k=0}^{n+1} w_{n+1,k}=1$. Then
the ratio $w_{n,k}/w_{n+1,k}$ is  decreasing in $k=0,\dots, n+1$,
which together with (\ref{eq4.69}) implies (\ref{eq4.6}) 
 (see \emph{e.g.} \cite[Chapter 1]{SS}). This completes the induction step and completes
the proof of (i).

(ii) By (\ref{eq3}),
\begin{align*}
\frac{\fh_{2n}}{\qh_{2n}}=2 \sum_{1\leq i <j\leq d} h_i h_j \frac{\MP(\Sh_{2n}=\bfe_d^{i}
+\bfe_d^j)}{\qh_{2n}}.
\end{align*}
Since by (i),
$\MP(\Sh_{2n}=\bfe_d^{i}+\bfe_d^j)/\qh_{2n} \leq \MP(\Sh_{2(n+1)}=\bfe_d^{i}+\bfe_d^j)/\qh_{2(n+1)}$,
(ii) follows. The proof is complete.
\end{proof}

\begin{lemma}\label{Lem9}
Let $\mathbf{h}$ be a $d$-dimensional probability vector.
\begin{align*} 
(i)&\;\; \text{For}\; n\geq 0 \;\text{and}\; 1\leq i<j\leq d, \qh_{2n}=
\MP(\Sh_{2n}=\zero_d)\geq
\MP(\Sh_{2n}=\bfe_d^{i}+\bfe_d^j).\\
(ii)&\;\;\text{For}\;n\geq 0,\\ 
&\qquad \qquad \qh_{2n}-\qh_{2n+2}\geq \frac{1}{2n+1}\; \qh_{2n+2}.\\
(iii)&\;\;\text{For}\;0\leq k \leq \frac{n}{2}-1,\\
&\qquad \qquad \frac{\qh_{2k}\; \qh_{2n-2k}}{\qh_{2k+2}\;\qh_{2n-2k-2}} \geq \frac{2k+2}{2k+1}\;\frac{2n-2k-1}{2n-2k}.
\\
(iv)& \;\; \text{For}\; n=2m\geq 2\;\text{and}\; 0\leq k<m, \\
&\qquad \qquad \frac{\qh_{2k}\;\qh_{4m-2k}}{(\qh_{2m})^2}\geq \frac{\binom{2k}{k}\binom{4m-2k}{2m-k}}{\binom{2m}{m}^2}.\\
(v)& \;\; \text{For}\; n=2m+1\geq 3\;
\text{and}\; 0\leq k<m,\\
&\qquad \qquad \frac{\qh_{2k}\;\qh_{4m-2k+2}}{\qh_{2m}\;\qh_{2m+2}}\geq 
\frac{\binom{2k}{k}\binom{4m-2k+1}{2m-k}}{\binom{2m}{m}\;\binom{2m+1}{m}}.
\end{align*}
\end{lemma}
\begin{proof}[Proof of Lemma \ref{Lem9}]
(i) Without loss of generality, we assume $(i,j)=(1,2)$. We prove
$\qh_{2n}=\MP(\Sh_{2n}=\zero_d)\geq
\MP(\Sh_{2n}=\bfe_d^{1}+\bfe_d^2)$ by induction on $d$.
The case $d=2$ is done by Lemma \ref{Lem5}(i).
Suppose $\qh_{2n}=\MP(\Sh_{2n}=\zero_d)\geq
\MP(\Sh_{2n}=\bfe_d^{1}+\bfe_d^2)$ for all $d$-dimensional probability vectors $\mathbf{h}$.
By (\ref{eq4.701}) and (\ref{eq4.71}), for a $(d+1)$-dimensional probability vector $\mathbf{h}$,
\begin{align*}
\MP(\Sh_{2n}=\bfe_{d+1}^{1}+\bfe_{d+1}^2)
&=\sum_{k=0}^n b_{n,n-k} \;\qha_{2k} \;\frac{\MP(\Sha_{2k}=\bfe_d^{1}+\bfe_d^2)}{\qha_{2k}}\\
&\leq \sum_{k=0}^n b_{n,n-k} \;\qha_{2k}=\MP(\Sh_{2n}=\zero_{d+1})=\qh_{2n},
\end{align*}
where  $b_{n,k}$ is given in (\ref{eq4.70}) and the inequality follows from the induction hypothesis applied to
  $\mathbf{h}'=(h_1,\dots, h_d)/(1-h_{d+1})$. The induction proof is complete.

(ii) 
By Lemma \ref{Lem1}  and $\qh_2=\sum_{i=1}^d h_i^2/2$,
\begin{align*}
\frac{2n+2}{2n+1}\;\qh_{2n+2}&=2\qh_2\;\qh_{2n}+\fh_{2n}\\
&=(\sum_{i=1}^d h_i^2) \; \qh_{2n}+2(\sum_{1\leq i<j\leq d} h_i h_j)\; \qh_{2n}-
2\sum_{1\leq i<j\leq d}h_i h_j (\qh_{2n}-\MP(\Sh_{2n}=\bfe_d^{i}+\bfe_d^j))\\
&\leq (\sum_{i=1}^d h_i^2 + 2\sum_{1\leq i<j\leq d}h_i h_j)\; \qh_{2n}=\qh_{2n},
\end{align*}
where the inequality is by (i), proving (ii).

(iii) The proof of (iii) closely follows that of Lemma \ref{Lem5}(iii). The only difference is
that the latter and the former call for
Lemma \ref{Lem4}(iii) and  Lemma \ref{Lem8}(ii), respectively.

(iv)--(v) The proofs of (iv) and (v) closely follow those of Lemma \ref{Lem5}(iv) and \ref{Lem5}(v).
The only difference is that the latter and the former call for Lemma \ref{Lem5}(iii) and
Lemma \ref{Lem9}(iii), respectively.
\end{proof}

\begin{proof}[Proof of Lemma \ref{Lem7}]
The proof closely follows that of Lemma \ref{Lem3} where two combinatorial identities in
(\ref{eq20.10}) and (\ref{eq20.16}) are derived which are needed for the proof
of Lemma \ref{Lem7}. We sketch the proof of $\gh_{n-1}\geq \gh_n$. For $n=1$, we have
\begin{align*}
\gh_1=\sum_{k=0}^1 \qh_{2k}\;\qh_{2-2k}=2 \qh_0\;\qh_2=2\;\qh_2=\sum_{i=1}^d h_i^2\leq 1=\gh_0,
\end{align*}
proving the case $n=1$.
For the case $n=2m\geq 2$, the proof of $\gh_{2m-1}\geq \gh_{2m}$ is essentially the same as
that of $\ga_{2m-1}\geq \ga_{2m}$. The only difference is that the latter and the former
call for Lemma \ref{Lem5}(ii) and 5(iv) and Lemma \ref{Lem9}(ii) and 9(iv), respectively. (Note that
the identity in (\ref{eq20.10}) is needed for the case $n=2m\geq 2$.)
For the case $n=2m+1\geq 3$, the proof of $\gh_{2m}\geq \gh_{2m+1}$ is essentially the same as
that of $\ga_{2m}\geq \ga_{2m+1}$. The only difference is that the latter and the former
call for Lemma \ref{Lem5}(ii) and 5(v) and Lemma \ref{Lem9}(ii) and 9(v), respectively.
(Note that the identity in (\ref{eq20.16}) is needed for the case $n=2m+1\geq 3$.)
\end{proof}

\begin{remark}\label{section4}
By Lemma \ref{Lem6}, $q_{2n}^{\mathbf{h}'}>q_{2n}^{\mathbf{h}''}$ for $n\geq 1$
if $\mathbf{h}' \succ \mathbf{h}''$.
This is an extension of Remark \ref{section3} to $d\geq 3$.
\end{remark}

Shoou-Ren Hsiau, Department of Mathematics, National Changhua University of Education, Taiwan, ROC.
Email: srhsiau@cc.ncue.edu.tw

Ting-Yi Tsai, Department of Mathematics, National Changhua University of Education, Taiwan, ROC.
Email: ejiwum1019@gmail.com

Yi-Ching Yao, Institute of Statistical Science, Academia Sinica, Taiwan, ROC. Email: yao@stat.sinica.edu.tw

\end{document}